\documentclass{article}
\usepackage{graphicx} % Required for inserting images

\usepackage{geometry, amsthm, amsmath, amssymb, amscd, color, mathrsfs,  enumerate, verbatim, amsrefs}

\usepackage{hyperref}

\usepackage{courier}

\usepackage[all]{xy}

\newenvironment{myproof}[1][\proofname]{%
  \proof[Proof #1]%
}{\endproof}

\newtheorem{Prop}[equation]{Proposition}
\newtheorem{Thm}[equation]{Theorem}
\newtheorem{Lem}[equation]{Lemma}
\newtheorem{Cor}[equation]{Corollary}
\theoremstyle{definition}\newtheorem{Def}[equation]{Definition}
\newtheorem{Ex}[equation]{Example}
\newtheorem{Rem}[equation]{Remark}

\theoremstyle{definition}
\theoremstyle{definition}\newtheorem{Conj}[equation]{Conjecture}
\theoremstyle{definition}\newtheorem{Not}[equation]{Notation}
\theoremstyle{definition}\newtheorem{Constr}[equation]{Construction}

\newcommand{\Q}{\mathbb{Q}}

\newcommand{\oK}{\overline{K}}
\newcommand{\oD}{\overline{D}}
\newcommand{\Z}{\mathbb{Z}}

\newcommand{\HQ}{\mathbf{H}}

\newcommand{\Int}{\textnormal{Int}}   
\newcommand{\IntQ}{\Int_{\Q}}   

\newcommand{\Max}{\textnormal{Max}}
\newcommand{\Pruf}{Pr\"{u}fer}
\newcommand{\Spec}{\text{Spec}}
\newcommand{\OurD}{ADDB}

\newcommand{\smat}[4]{[\begin{smallmatrix} #1 & #2 \\ #3 & #4 \end{smallmatrix}]}

\newcommand{\Gal}{\text{Gal}}

\newcommand{\mfa}{\mathfrak{a}}
\newcommand{\mfI}{\mathfrak{I}}
\newcommand{\mfm}{\mathfrak{m}}
\newcommand{\mfp}{p}
\newcommand{\mfM}{\mathfrak{M}}
\newcommand{\msJ}{\mathscr{J}}

\newcommand{\msD}{\mathscr{D}}
\newcommand{\msE}{\mathscr{E}}
\newcommand{\mcD}{\mathcal{D}}

\newcommand{\mcR}{\mathcal{R}}

\newcommand{\bfi}{\mathbf{i}}
\newcommand{\bfj}{\mathbf{j}}
\newcommand{\bfk}{\mathbf{k}}
\newcommand{\bfh}{\mathbf{h}}

\numberwithin{equation}{section}

\title{A classification of Pr\"{u}fer domains of integer-valued polynomials on algebras}

\author{Giulio Peruginelli\footnote{Department of Mathematics ``Tullio Levi-Civita'' University of Padova, Via Trieste, 63 35121 Padova, Italy. gperugin@math.unipd.it}
\and
Nicholas J. Werner\footnote{Department of Mathematics, Computer and Information Science, State University of New York at Old Westbury, Old Westbury, NY 11560, USA. wernern@oldwestbury.edu}}
\date{\today}

\begin{document}

%\maketitle
\leftmark{\noindent  accepted for publication in the Bulletin of the London Mathematical Society (2026)}
{\let\newpage\relax\maketitle}

\begin{abstract}
\noindent Let $D$ be an integrally closed domain with quotient field $K$  and $A$ a torsion-free $D$-algebra that is finitely generated  as a $D$-module and such that $A\cap K=D$. We give a complete classification of those $D$ and $A$ for which the ring $\Int_K(A)=\{f\in K[X] \mid f(A)\subseteq A\}$ is a \Pruf{} domain. If $D$ is a semiprimitive domain, then we prove that $\Int_K(A)$ is Pr\"ufer if and only if $A$ is commutative and isomorphic to a finite direct product of almost Dedekind domains with finite residue fields, each of them satisfying  a double-boundedness condition on its ramification indices and residue field degrees.
\end{abstract}

{\small \noindent Keywords: integer-valued polynomial, Pr\"ufer domain, algebra, integrally closed domain, almost Dedekind domain.

\noindent MSC Primary 13F20, 13F05, 12J20,  13A18.\\}

\section{Introduction}

Let $D$ be an integrally closed domain with quotient field $K$. We will always assume that $D$ is properly contained in $K$. Let $A$ be a $D$-algebra which is a torsion-free and finitely generated $D$-module. Let $B=A\otimes_D K$  be the extended $K$-algebra; note that $B$ is equal to the localization $(D\setminus\{0\})^{-1}A$. We can embed $K$ and $A$ inside $B$ via the natural mappings $k \mapsto 1 \otimes k$ and $a \mapsto a \otimes 1$, and we will use this fact throughout this paper. For a subset $S$ of $A$ we define the ring of integer-valued polynomials over $S$ as:
$$\Int_K(S,A)=\{f\in K[X] \mid f(S)\subseteq A\}.$$
When $S=A$, we set $\Int_K(A,A)=\Int_K(A)$. We will also assume throughout that $A\cap K=D$ which amounts to say that $\Int_K(A)\subseteq\Int_K(D)=\Int(D)$, the ring of integer-valued polynomials over $D$.

Integer-valued polynomials have a long history in commutative ring theory. One problem that attracted significant interest \cites{Chabert1987, GilmPrufer, Chab1993, Loper1997} was to classify the domains $D$ such that $\Int(D)$ is a \Pruf{} domain. 

Necessary and sufficient conditions on $D$ for this to occur were determined in \cite{LopIntD}. Since the late 2000s, attention has turned to rings of the form $\Int_K(A)$ and $\Int_K(S,A)$ \cites{Werner2010, LopWer, EFJ2013, Frisch2013,   FrischCorr, Per1, EvrardJohnson2015, Frisch2017, HNR2020, Mulay, NagHaf2024}. Here, as in the traditional case where $A=D$, researchers have studied situations in which these rings are or are not \Pruf{}. For instance, when $D=\Z$ and $A$ is a ring of algebraic integers, it is known that $\Int_\Q(A)$ is \Pruf{} \cite[Theorem 3.7]{LopWer}. By contrast, if $D$ is a Dedekind domain with finite residue fields and $A$ is the matrix algebra $A=M_n(D)$ with $n \geq 2$, then $\Int_K(A)$ is never integrally closed \cite[Corollary 3.4]{PerWerPropInt}, hence cannot be \Pruf{}.

In this paper we give a complete answer to \cite[Problem 28]{CFSG}, namely, establishing exactly when $\Int_K(A)$ is a \Pruf{} domain. Since $\Int_K(A) \subseteq \Int(D)$ and overrings of \Pruf{} domains are \Pruf{}, in order for $\Int_K(A)$ to be \Pruf{} it is necessary that $\Int(D)$ be \Pruf{}. Recall that $D$ is an almost Dedekind domain if $D_M$ is a discrete valuation ring for each maximal ideal $M$ of $D$. Note that a Noetherian almost Dedekind domain is a Dedekind domain. In order for $\Int(D)$ to be \Pruf{}, it is necessary (but not sufficient) for $D$ to be an almost Dedekind domain with finite residue fields \cite[Proposition 6.3]{Chabert1987}. For such a domain, we adopt the following notations and vocabulary.

\begin{Not}\label{D_0 notations}
Let $D$ be an almost Dedekind domain with finite residue fields. 

\begin{itemize}

\item The domain $D$ contains a principal ideal domain $D_0$, the structure of which depends on the characteristic of $K$. If $\text{char}(K)=0$, then $D_0 \cong \Z$. If $\text{char}(K)$ is positive, then $D_0 \cong F[t]$, the univariate polynomial ring over the finite field $F$ with $\text{char}(K)$ elements.

\item Let $\mfp$ denote a nonzero prime element of $D_0$. Explicitly, if $D_0 \cong \Z$, then $\mfp$ is an integral prime; otherwise, $\mfp$ is a polynomial irreducible over $F$.

\item For each prime $P$ of $D$, let $v_P$ be the normalized $P$-adic valuation on $K$. 

\item For each $\mfp$ and each $P$ over $\mfp$, define
\begin{equation*}
\quad e_{P,\mfp} = v_P(\mfp) \quad \text{ and } \quad f_{P,\mfp} = [(D/P):(D_0/\mfp D_0)].
\end{equation*}

\item For each $\mfp$, define
\begin{equation*}
E_\mfp := \{ e_{P,\mfp} \mid P|\mfp D_0\} \quad  \text{ and } \quad F_\mfp := \{ f_{P,\mfp} \mid P|\mfp D_0\}.
\end{equation*}

\item We say that $D$ satisfies the \textit{double-boundedness condition} if, for each $\mfp$, both $E_\mfp$ and $F_\mfp$ are bounded.

\end{itemize}
\end{Not}

\begin{Def}\label{def: double-boundedness}
A domain $D$ will be called an \OurD-domain (short for ``almost Dedekind, doubly bounded'') if $D$ is an almost Dedekind domain with finite residue fields that satisfies the double-boundedness condition from Notation \ref{D_0 notations}. 
\end{Def}

By the following theorem, \OurD-domains are exactly those for which $\Int(D)$ is \Pruf{}.

\begin{Thm}\label{thm: loper}(\cite[Theorem 2.5]{LopIntD})
Let $D$ be an integral domain that is not a field. Then, $\Int(D)$ is \Pruf{} if and only if $D$ is an \OurD{}-domain.
\end{Thm}

Thus, in order for $\Int_K(A)$ to be \Pruf{}, it is necessary that $D$ be an \OurD{}-domain. Since \Pruf{} domains are integrally closed, we can obtain further information on when $\Int_K(A)$ is \Pruf{} by studying the integral closure of $\Int_K(A)$. For this, the elements of $A$ that are integral over $D$ are prominent. The fact that $A$ is finitely generated as a $D$-module implies that $B = A \otimes_D K$ is finitely generated as a $K$-module. Hence, each element of $B$ is algebraic over $K$.

\begin{Def}\label{def: integral elements}
We denote by $\oD$ the absolute integral closure of $D$ in a fixed algebraic closure $\oK$ of $K$.  For an element $b\in B$, we let $\mu_b\in K[X]$ be the minimal monic polynomial of $b$ over $K$ and let $\Omega_b$ be the set of roots of $\mu_b$ inside $\oK$. When $S \subseteq B$, we set $\Omega_S=\bigcup_{b\in S}\Omega_b$. We say that $b$ is \textit{integral over $D$} if $b$ solves a monic polynomial over $D$, we let $A':=\{b\in B\mid b \text{ is integral over } D\}$, and we say that $A$ is \textit{integrally closed} if $A=A'$.
\end{Def}

Our assumptions on $A$ imply that $A \subseteq A'$. However, there is no guarantee that $A$ is equal to $A'$, and in fact $A'$ may fail to be a ring because $A$ is allowed to be noncommutative (this behavior is similar to that of $D$-orders in $K$-algebras, cf.\ \cite[Section 8]{Reiner}). Nevertheless, in some cases the set $A'$ affords a description of the integral closure of $\Int_K(A)$. 

\begin{Def}\label{def: Int_K(A,A')}
Let $\Int_K(A,A') := \{f \in K[X] \mid f(A) \subseteq A'\}$. For each $n \geq 1$, let
\begin{equation*}
\Lambda_n = \{\alpha \in \oK \mid \alpha \text{ solves a monic polynomial in $D[X]$ of degree $n$}\}.
\end{equation*}
For $\msE \subseteq \Lambda_n$, let $\Int_K(\msE, \Lambda_n) := \{f \in K[X] \mid f(\msE) \subseteq \Lambda_n\}$.
\end{Def}

\begin{Thm}\label{thm: Int_K(A,A')}(\cite[Theorem 10 \& 13]{PerWer})
Let $D$ be an integrally closed domain. Let $n$ be the vector space dimension of $B$ over $K$. Then, $\Int_K(A,A')$ is an integrally closed ring, and $\Int_K(A,A') = \Int_K(\Omega_A, \Lambda_n)$. Moreover, if all residue rings (by nonzero ideals) of $D$ are finite, then $\Int_K(A,A')$ is equal to the integral closure of $\Int_K(A)$.
\end{Thm}

Note that if $D$ is a Dedekind domain with finite residue fields, then Theorem \ref{thm: Int_K(A,A')} applies and $\Int_K(A,A')$ is the integral closure of $\Int_K(A)$. It is not known whether the same is true when $D$ is a general \OurD{}-domain (i.e.\ if $D$ is not Noetherian).

We now state our main result, which says in part that $\Int_K(A)$ is \Pruf{} if and only if $A=A'$. The theorem also provides structural descriptions of $A$ and $B$ that are necessary and sufficient for $\Int_K(A)$ to be \Pruf{}. Recall that for a $D$-module $M$, $M_P$ denotes the localization of $M$ at $D \setminus P$. Also, recall that a ring $R$ is semiprimitive when the Jacobson radical of $R$ is $\{0\}$.

\begin{Thm}\label{thm: main}
Let $D$ be an \OurD-domain. Let $A$ be a $D$-algebra that is torsion-free and finitely generated as a $D$-module, and such that $A \cap K = D$. Let $B = A \otimes_D K$. 
\begin{enumerate}[(1)]
\item The following are equivalent.
\begin{enumerate}[(i)]
\item $\Int_K(A)$ is \Pruf{}.
\item $\Int_K(A) = \Int_K(A,A')$.
\item $A = A'$.
\item For all $a \in A$, $A \cap K[a]$ is integrally closed in $K[a]$.
\item For some $t \geq 1$, $B \cong \prod_{i=1}^t \mcD_i$, where each $\mcD_i$ is a division ring over $K$, and $A \cong \prod_{i=1}^t A_i$, where each $A_i$ is an integrally closed domain (not necessarily commutative) in $\mcD_i$.
\item For each prime ideal $P$ of $D$, $\Int_K(A)_P$ is \Pruf{}.
\item For each prime ideal $P$ of $D$, $\Int_K(A_P)$ is \Pruf{}.
\item For each prime ideal $P$ of $D$, $\Int_K(A_P)$ is integrally closed.
\end{enumerate}

\item Assume $D$ is semiprimitive. Then, $\Int_K(A)$ is \Pruf{} if and only if $A \cong \prod_{i=1}^t A_i$, where $A_i$ is the integral closure of $D$ in a finite field extension $F_i/K$. In this case, each $A_i$ is an \OurD{}-domain, and $A$ is a commutative ring.
\end{enumerate}
\end{Thm}

Each condition in Theorem \ref{thm: main}(1) implies that $\Int_K(A)$ is integrally closed. In analogy with condition (viii), we suspect that $\Int_K(A)$ is \Pruf{} if and only if $\Int_K(A)$ is integrally closed; however, we have only been able to prove this with additional assumptions on $D$ or $\Int_K(A)$.  

\begin{Conj}\label{conj: main}
Let $D$ and $A$ be as in the statement of Theorem \ref{thm: main}. If $\Int_K(A)$ is integrally closed, then $\Int_K(A)$ is \Pruf{}.
\end{Conj}

The full resolution of Conjecture \ref{conj: main} remains an open problem. We can show that the conjecture is true if $D$ is an \OurD{}-domain that is also a Dedekind domain or, more generally, if $\Int_K(A)$ is well-behaved with respect to localization at primes of $D$:

\begin{Thm}\label{thm: conjecture special cases}
Let $D$ and $A$ be as in the statement of Theorem \ref{thm: main}. If $\Int_K(A)_P = \Int_K(A_P)$ for all primes $P \subset D$, then Conjecture \ref{conj: main} is true.
\end{Thm}

In particular, Theorem \ref{thm: conjecture special cases} holds when $D$ is Dedekind, because in that case $\Int_K(A)_P = \Int_K(A_P)$ by \cite[Lemma 3.2]{PerWerNontrivial}.

The paper is organized as follows. In Section \ref{sec: finite subsets}, we classify the finite subsets $S \subseteq A$ such that $\Int_K(S,A)$ is \Pruf{}. This gives us some control on the overrings of $\Int_K(A)$, and these are used in Section \ref{sec: necessary} where we derive a number of necessary conditions on $A$ in order for $\Int_K(A)$ to be \Pruf{}. In Theorem \ref{thm: A is a direct product}, we prove that if $\Int_K(A)$ is \Pruf{}, then $A$ and $B$ have the structure specified in condition (v) of Theorem \ref{thm: main}(1). Moreover, we prove Theorem \ref{thm: main}(2) (see Corollary \ref{cor: D semiprimitive}), which shows that if $D$ is semiprimitive, then $A$ must be commutative. If $D$ is not semiprimitive, then it is possible that $\Int_K(A)$ is \Pruf{} and $A$ is noncommutative; an example of such an algebra $A$ is given in Section \ref{sec: quaternions} at the end of the paper. We complete the proof of Theorem \ref{thm: main} in Section \ref{sec: sufficient} by constructing a family of \Pruf{} domains contained in $K[X]$. We show that $\Int_K(A,A')$ must contain one of these \Pruf{} domains, and hence is itself \Pruf{} (Corollary \ref{cor: Lambda_n is Prufer}). Under any of conditions (iii) -- (viii) in Theorem \ref{thm: main}(1), we demonstrate that $\Int_K(A) = \Int_K(A,A')$, and conclude $\Int_K(A)$ is \Pruf{}. Section \ref{sec: sufficient} closes with a discussion of Conjecture \ref{conj: main} and Theorem \ref{thm: conjecture special cases}.

\section{Finite subsets of $A$}\label{sec: finite subsets}

Our first goal is to determine when $\Int_K(S,A)$ is \Pruf{} for a finite subset $S \subseteq A$. Since \Pruf{} domains are integrally closed, we begin by studying the integral closure of $\Int_K(S,A)$. In \cite{PerFinite}, subsets of $\oD$ were used to describe the integral closure of $\Int_K(S,A)$. Recall from Definition \ref{def: integral elements} that for $b \in B$ and $S \subseteq B$, we let $\Omega_b \subseteq \oK$ be the set of roots of $\mu_b$ in $\oK$, and $\Omega_S = \bigcup_{b \in S} \Omega_b$.

\begin{Thm}\label{integral closure}
Let $D$ be an integrally closed domain.
\begin{enumerate}[(1)]
\item (\cite[Proposition 4.2]{PerFinite}) Let $\Omega$ be a finite subset of $\oD$. Then, $\Int_K(\Omega, \oD)$ is \Pruf{} if and only if $D$ is \Pruf{}.

\item (\cite[Corollary 1.5]{PerFinite}) Let $S$ be a finite subset of $A$. Then, the integral closure of $\Int_K(S,A)$ is equal to $\Int_K(\Omega_S,\oD)=\{f\in K[X] \mid f(\Omega_S)\subseteq\oD\}$. Thus, $\Int_K(S,A)$ is \Pruf{} if and only if $D$ is \Pruf{} and $\Int_K(S,A)$ is integrally closed.
\end{enumerate}
\end{Thm}

It remains to determine conditions on $S$ and $A$ under which $\Int_K(S,A)$ is integrally closed. We begin by considering the case where $S =\{a\}$ for some $a \in A$.  Even when $S$ is a singleton set, it is not always apparent whether or not $\Int_K(S,A)$ will be integrally closed. Consider the following two examples involving $2 \times 2$ matrices over $\Z$.

\begin{Ex}\label{matrix ex 1} Let $D = \Z$, $A = M_2(\Z)$, and $\mu(X) = X^2-2X-4$, with roots $1 \pm \sqrt{5}$. Let $\theta = \tfrac{1+\sqrt{5}}{2}$. Let $a = \smat{0}{4}{1}{2}$ and $b=\smat{0}{2}{2}{2}$, both of which are matrices in $A$ with minimal polynomial $\mu$. Since $a$ and $b$ share the same minimal polynomial, they are conjugate over $GL(2,\Q)$, although one may check that they are not conjugate over $GL(2,\Z)$. We will show that $\Int_\Q(\{a\},A)$ is not integrally closed, but $\Int_\Q(\{b\},A)$ is integrally closed.

Let $\mcR = \Int_\Q(\{1 \pm \sqrt{5}\}, \Z[\theta])$. Then, $\mcR$ is the integral closure of both $\Int_\Q(\{a\},A)$ and $\Int_\Q(\{b\},A)$ by Theorem \ref{integral closure}. To see that $\Int_\Q(\{a\},A)$ is not integrally closed, note that $\tfrac{X}{2} \in \mcR$, but $\tfrac{a}{2} \notin A$. Hence, $\Int_\Q(\{a\},A) \subsetneqq \mcR$.

Next, we show that $\Int_\Q(\{b\},A) = \mcR$. We have $\Int_\Q(\{b\},A) \subseteq \mcR$ by Theorem \ref{integral closure}. For the reverse containment, let $f \in \mcR$. Then, $f(1+\sqrt{5}) \in \Z[\theta]$, so there exist $c_0, c_1 \in \Z$ such that $f(1+\sqrt{5}) = c_0 + c_1 \theta$. Note that $\Q(b) \cong \Q[X]/(\mu(X)) \cong \Q(1+\sqrt{5})$, so there exists a ring isomorphism $\phi: \Q(1+\sqrt{5}) \to \Q(b)$ such that $\phi(1+\sqrt{5}) = b$. As $\phi(\theta) = \tfrac{b}{2} = \smat{0}{1}{1}{1}$ and polynomials in $\Q[X]$ commute with $\phi$, we have
\begin{equation*}
f(b) = f(\phi(1+\sqrt{5})) = \phi(f(1+\sqrt{5})) = c_0 + c_1\smat{0}{1}{1}{1} \in A.
\end{equation*}
This shows that $f \in \Int_\Q(\{b\},A)$. Thus, $\Int_\Q(\{b\},A) = \mcR$ is integrally closed.
\end{Ex}

\begin{Ex}\label{matrix ex 2}
Again let $D = \Z$ and $A = M_2(\Z)$. Let $k \in \Z$, $k > 0$. Let $a=\smat{0}{k}{k}{0}$ and $b=\smat{k}{0}{0}{-k}$. Then, $\mu(X)=X^2-k^2$ is the minimal polynomial of both $a$ and $b$. As in the previous example, $a$ and $b$ are conjugate over $GL(2,\Q)$, but are not conjugate over $GL(2,\Z)$. We demonstrate that $\Int_\Q(\{a\},A)$ is not integrally closed, but $\Int_\Q(\{b\},A)$ is integrally closed.

This time, the integral closure of both $\Int_\Q(\{a\},A)$ and $\Int_\Q(\{b\},A)$ is $\mcR=\Int_\Q(\{\pm k\}, \Z)$. We have $\tfrac{X-k}{2k} \in \mcR \setminus \Int_\Q(\{a\},A)$, because
\begin{equation*}
\tfrac{1}{2k}(a-k) = \begin{bmatrix} -k/2 & k/2 \\ k/2 & -k/2 \end{bmatrix} \notin A.
\end{equation*}
Thus, $\Int_\Q(\{a\},A)$ is not integrally closed. However, since $b$ is a diagonal matrix, for each $f \in \mcR$ we have $f(b) = \smat{f(k)}{0}{0}{f(-k)}$, and this is in $A$ because $f(\pm k) \in \Z$. Hence, $\Int_\Q(\{b\},A) = \mcR$.
\end{Ex}

One distinguishing feature between $a$ and $b$ in these examples is the structure of the algebras $A \cap K[a]$ and $A \cap K[b]$. For each $a \in A$, $K[a]$ is a commutative $K$-subalgebra of $B$ and the intersection $A\cap K[a]$ is a $D$-subalgebra of $K[a]$. Recall that a commutative ring $R$ (possibly with zero divisors) is said to be integrally closed if it is integrally closed in its total ring of fractions $T$ (i.e.\ every element of $T$ which satisfies a monic polynomial over $R$ belongs to $R$).

\begin{Lem}\label{A cap K[a] integral closure}
Let $a \in A$, let $R = A \cap K[a]$, and let $T$ be the total ring of fractions of $R$. 
\begin{enumerate}[(1)]
\item $T = K[a]$.
\item $R$ is integrally closed if and only if $R$ is equal to the integral closure of $D$ in $K[a]$.
\end{enumerate}
\end{Lem}
\begin{proof}
(1) Recall that $\mu_a \in D[X]$ is the minimal polynomial of $a$ over $D$. Factor $\mu_a$ over $K$ as $\mu_a(X) = \prod_{i=1}^t p_i(X)^{e_i}$, where each $p_i \in K[x]$ is irreducible and each $e_i \geq 1$. 
Then,
\begin{equation*}
K[a] \cong K[x]/(\mu_a(X)) \cong \prod_{i=1}^t K[x]/(p_i(X)^{e_i}).
\end{equation*}
Each ring $K[x]/(p_i(X)^{e_i})$ is local with maximal ideal generated by $p_i(X)$. In each such ring, the set of zero divisors is equal to the maximal ideal, and all other elements are units. Hence, each of these rings is equal to its own total ring of fractions. It follows that $K[a]$ is also equal to its own total ring of fractions. Now, $D[a] \subseteq R \subseteq K[a]$, and considering the total rings of fractions of these rings, we see that $K[a] \subseteq T \subseteq K[a]$. Thus, $T=K[a]$.

(2) Let $R'$ be the integral closure of $R$ in $T$, and let $D'$ be the integral closure of $D$ in $K[a]$. We have $D \subseteq R$, and $T=K[a]$ by (1), so $D' \subseteq R'$. Conversely, let $b \in R' \subseteq K[a]$. Then, there exists a monic $f \in R[X] \subseteq A[X]$ such that $f(b) = 0$. Since each coefficient of $f$ is in $A$ and $A$ is integral over $D$, $b$ is integral over $D$. Hence, $b \in D'$ and $R' \subseteq D'$.
\end{proof}

In both Example \ref{matrix ex 1} and Example \ref{matrix ex 2}, the ring $A \cap K[a]$ is not integrally closed. In the former example, $a/2 \in K[A] \setminus A$ is integral over $D$, and in the latter $(a-k)/2k \in K[a] \setminus A$ is integral over $D$. As we now proceed to show, $\Int_K(\{a\},A)$ is integrally closed if and only if $A \cap K[a]$ is integrally closed. 

\begin{Lem}\label{lem: eigenvalues}
Let $b \in B$. 
\begin{enumerate}[(1)]
\item $b$ is integral over $D$ if and only if $\Omega_b \subseteq \oD$.
\item For each $f \in K[X]$, we have $\Omega_{f(b)} = f(\Omega_b)$.
\end{enumerate}
\end{Lem}
\begin{proof}
Let $n$ be the vector space dimension of $B$ over $K$. Then, $B$ embeds into the matrix ring $M_n(K)$, and under this embedding $K$ corresponds to the field of scalar matrices in $M_n(K)$. Thus, each $b \in B$ corresponds to a matrix $M_b \in M_n(K)$ with minimal polynomial $\mu_b \in K[X]$. The eigenvalues of $M_b$ are exactly the roots of $\mu_b$, and it is shown in \cite[Lemma 2]{PerWer} that $\mu_b \in D[X]$ if and only if each eigenvalue of $M_b$ is in $\oD$. This proves (1). For (2), we use the fact that for a matrix $C \in M_n(K)$, a scalar $\lambda \in K$, and a polynomial $f \in K[X]$, $\lambda$ is an eigenvalue of $C$ if and only if $f(\lambda)$ is an eigenvalue of $f(C)$. 
\end{proof}

\begin{Prop}\label{A cap K[a] and IVP}
Let $a \in A$. The following are equivalent.
\begin{enumerate}[(1)]
\item $\Int_K(\{a\},A)$ is integrally closed.
\item $\Int_K(\{a\},A) = \Int_K(\Omega_a, \overline{D})$.
\item $A \cap K[a]$ is integrally closed.
\end{enumerate}
\end{Prop}
\begin{proof}

(1) $\Leftrightarrow$ (2) This follows from Theorem \ref{integral closure}, because $\{a\}$ is finite.

(3) $\Rightarrow$ (2) Assume that $A \cap K[a]$ is integrally closed. From Theorem \ref{integral closure}, $\Int_K(\Omega_a, \overline{D})$ is the integral closure of $\Int_K(\{a\},A)$, and so it is always true that $\Int_K(\{a\},A) \subseteq \Int_K(\Omega_a, \overline{D})$. Let $f \in \Int_K(\Omega_a, \overline{D})$. Then, $f(a) \in K[a]$, because $f \in K[X]$, and $f(a)$ is integral over $D$ by Lemma \ref{lem: eigenvalues} because $\Omega_{f(a)} = f(\Omega_a) \subseteq \overline{D}$ by Lemma \ref{lem: eigenvalues}. By Lemma \ref{A cap K[a] integral closure}, $A \cap K[a]$ is the integral closure of $D$ in $K[a]$. Thus, $f(a) \in A \cap K[a] \subseteq A$, $f \in \Int_K(\{a\},A)$, and $\Int_K(\{a\},A) = \Int_K(\Omega_a, \overline{D})$.

(1) $\Rightarrow$ (3) Assume that $\Int_K(\{a\},A)$ is integrally closed. Let $b \in K[a]$ be integral over $D$. Then, $b = f(a)$ for some $f \in K[X]$, and there exists a monic $g \in D[X]$ such that $g(b) = 0$. So, $g \circ f \in \Int_K(\{a\},A)$. Since $g \in D[X] \subseteq \Int_K(\{a\},A)$, this implies that $f$ is integral over $\Int_K(\{a\},A)$, and hence $f  \in \Int_K(\{a\},A)$. Thus, $b = f(a) \in A \cap K[a]$. This shows that $A \cap K[a]$ is the integral closure of $D$ in $K[a]$, so $A \cap K[a]$ is integrally closed by Lemma \ref{A cap K[a] integral closure}.
\end{proof}

We now have necessary and sufficient conditions on $A$ and $D$ to describe when $\Int_K(\{a\},A)$ is \Pruf{}. The conditions extend naturally to the case where $S \subseteq A$ is finite.

\begin{Thm}\label{singleton Prufer}\label{pointwise Prufer domain}
Let $a \in A$. The following are equivalent.
\begin{enumerate}[(1)]
\item $\Int_K(\{a\},A)$ is \Pruf{}.
\item $D$ is \Pruf{} and $\Int_K(\{a\},A)$ is integrally closed.
\item $D$ is \Pruf{} and $A \cap K[a]$ is integrally closed.
\end{enumerate}
\end{Thm}
\begin{proof}
Apply Theorem \ref{integral closure} and Proposition \ref{A cap K[a] and IVP}.
\end{proof}

\begin{Cor}
   Let $D$ be a \Pruf{} domain and let $S\subset A$ be a finite subset. Then $\Int_K(S,A)$ is \Pruf{} if and only if for each $a\in S$, $\Int_K(\{a\}, A)$ is \Pruf{}. 
\end{Cor}
\begin{proof}
   We have just to observe that 
   $$\Int_K(S,A)=\bigcap_{a\in S}\Int_K(\{a\}, A).$$
 If $\Int_K(S,A)$ is \Pruf{} then each of its overrings $\Int_K(\{a\}, A)$ is \Pruf{} as well.  Conversely, if $\Int_K(\{a\}, A)$ is \Pruf{}, then by Theorem \ref{integral closure}, $\Int_K(\{a\}, A)=\Int_K(\Omega_a,\oD)$. The intersection of these rings over all $a\in S$ is equal to $\Int_K(\Omega_S,\oD)$ which is \Pruf{} because $\Omega_S=\bigcup_{a\in S}\Omega_a$ is finite.
\end{proof}

\section{Necessary conditions}\label{sec: necessary}

In this section, we will establish necessary conditions for $\Int_K(A)$ to be \Pruf{}. Since one of our standard assumptions on $A$ is that $A \cap K = D$, we have $\Int_K(A) \subseteq \Int(D)$. Thus, if $\Int_K(A)$ is \Pruf{}, then so is $\Int(D)$. If the reader wishes, for the remainder of this section they may consider $D$ to be an \OurD-domain (i.e.\ $D$ is such that $\Int(D)$ is \Pruf), since this is the assumption we will make in our major theorems. However, we state our intermediary results using only the properties of the integrally closed domain $D$ that are needed to reach the desired conclusion. In fact, a number of these results (Lemmas \ref{lem: D-topology is Hausdorff}, \ref{lem: a not in pA}, \ref{lem: max spectrum}, \ref{lem: extending maximal ideals}; Propositions \ref{prop: when J(D)=0}, \ref{prop: when A is a direct product}; Theorem \ref{thm: IntK(A) not Prufer}) hold even when $D$ is not integrally closed.

Recall that an element $b \in B$ is integral over $D$ if $b$ solves a monic polynomial in $D[X]$, and that $A$ is integrally closed if $A=A'=\{b\in B\mid b \text{ is integral over }D\}$. While $A'$ may not be a ring, $A \cap K[a]$ is a commutative subring of $K[a]$ for each $a \in A$. So, the integral closure of $A \cap K[a]$ is a commutative ring, and these rings allow us to describe when $A=A'$.

\begin{Lem}\label{pointwise integrally closed} 
The ring $A$ is integrally closed if and only if, for each $a\in A$, $A\cap K[a]$ is integrally closed.
\end{Lem}
\begin{proof}
Clearly, if $A=A'$ and $a\in A$, then every element $a'$ of $K[a]$ that is integral over $A\cap K[a]$ is also integral over $D$. Hence, $a' \in A$ and $A\cap K[a]$ is integrally closed. Conversely, suppose that $A\cap K[a]$ is integrally closed for each $a\in A$ and let $a'\in A'$. Since $a'\in B=(D\setminus \{0\})^{-1}A$, there exist  $a\in A$ and $d\in D$ such that $a'=\frac{a}{d}$. So, $a'\in K[a]$ and is integral over $D\subset D[a]\subseteq A\cap K[a]$. It follows that $a'\in A\cap K[a]\subset A$.
\end{proof}

 Recall that a ring $R$ is reduced if it has no nonzero nilpotent elements. One obstruction to $\Int_K(A)$ being \Pruf{} is the presence of nonzero nilpotent elements of $A$. In the next two lemmas, we show that if $\Int_K(A)$ is \Pruf{}, then $A$ must be reduced. We thank Hwankoo Kim for pointing out Lemma \ref{lem: D-topology is Hausdorff}.

\begin{Lem}\label{lem: D-topology is Hausdorff}
The $D$-topology on $A$ is Hausdorff. That is, $\bigcap_{d\in D\setminus\{0\}}dA=\{0\}$.
\end{Lem}
\begin{proof}
Let $I=\bigcap_{d\in D\setminus\{0\}}dA$. Then, $I$ is a divisible $D$-submodule of $A$ because $dI = I$ for all nonzero $d \in D$. Hence, $I$ is injective because $A$ is torsion-free \cite[Theorem 2.4.5(3)]{KimWang}, and so $I$ is a direct summand of $A$. Thus, $I$ is a finitely generated $D$-submodule of $A$. Let $d \in D$ such that $d$ is not a unit of $D$, which exists because $D$ is not a field. By \cite[Corollary 2.5]{AtiyahMacDonald}, there exists $c \in D$ such that $(1+cd)I = \{0\}$. Since $d$ is not a unit of $D$, $1+cd \ne 0$. We conclude that $I = \{0\}$.
\end{proof}

\begin{Lem}\label{K[a] reduced is not needed}\mbox{}
\begin{enumerate}[(1)]
\item Let $a \in A$. If $A \cap K[a]$ is integrally closed, then $K[a]$ is reduced.
\item If $\Int_K(A)$ is \Pruf{}, then both $A$ and $B$ are reduced.
\end{enumerate}
\end{Lem}
\begin{proof}
(1) We prove the contrapositive. Suppose that $K[a]$ is not reduced. Then, there exists a nonzero $b \in K[a]$ and $n \geq 2$ such that $b^n = 0$. Let $f \in K[X]$ be such that $b = f(a)$. Write $f(X) = g(X)/d$, where $g \in D[X]$ and $d \in D$. Since $f(a) \ne 0$, $g(a)$ is also nonzero. So, there exists a nonzero $c \in D$ such that $g(a) \notin cA$ by Lemma \ref{lem: D-topology is Hausdorff}. Then, the polynomial $g(x)/c \notin \Int_K(\{a\},A)$. However, $f(a)^n = 0$, so
\begin{equation*}
\Big(\dfrac{g(X)}{c}\Big)^n = \Big(\dfrac{df(X)}{c}\Big)^n \in \Int_K(\{a\},A).
\end{equation*}
Thus, $\Int_K(\{a\},A)$ is not integrally closed. By Proposition \ref{A cap K[a] and IVP}, $A \cap K[a]$ is not integrally closed.

(2) Suppose that $\Int_K(A)$ is \Pruf{}, but $A$ is not reduced. Let $a \in A$ be nonzero and nilpotent. Since $\Int_K(A)$ is \Pruf{}, its overring $\Int_K(\{a\},A)$ is also \Pruf{}, and hence is integrally closed. But, $K[a]$ is not reduced, so by (1), $A \cap K[a]$ is not integrally closed. By Proposition \ref{A cap K[a] and IVP}, $\Int_K(\{a\},A)$ is not integrally closed, which is a contradiction. Thus, $A$ must be reduced, and it follows that $B$ is also reduced.
\end{proof}

\begin{Thm}\label{thm: A is a direct product}
    Suppose that $\Int_K(A)$ is a \Pruf{} domain. Then, $A=A'$, and $B$ and $A$ have the following structure:
    \begin{itemize}
        \item For some $t \geq 1$, $B\cong\prod_{i=1}^t \mcD_i$, where $\mcD_i$ is a division ring over $K$ for each $i=1,\ldots,t$, and
        \item $A\cong \prod_{i=1}^t A_i$, where $A_i\subset \mathcal{D}_i$ is an integrally closed domain (but is not necessarily commutative) for each $i=1,\ldots,t$.
    \end{itemize}
\end{Thm}
\begin{proof}
Since $\Int_K(A)$ is \Pruf{}, each overring of the form $\Int_K(\{a\},A)$, where $a\in A$, is also \Pruf{}. By Theorem \ref{pointwise Prufer domain}, $A\cap K[a]$ is integrally closed for each $a\in A$. So, $A=A'$ by Lemma \ref{pointwise integrally closed}. Also, both $A$ and $B$ are reduced by Lemma \ref{K[a] reduced is not needed}.

Now, since $B$ has finite dimension over $K$ it is artinian (see e.g.\ \cite[\S 23, p. 158]{CR}), and $B$ contains no nonzero nilpotent ideal because it is reduced. By Wedderburn's theorems \cite[Theorem 7.1, Theorem 7.4]{Reiner}, $B$ is semisimple, thus isomorphic to a finite direct product of matrix rings $M_{n_i}(\mcD_i)$, where $\mcD_i$ is a division ring over (an isomorphic copy of) $K$ for each $i$. If any $n_i \geq 2$, then $M_{n_i}(\mcD_i)$---and hence $B$---contains nonzero nilpotent elements, which is impossible because $B$ is reduced. Hence, $B \cong \prod_{i=1}^t \mcD_i$ for some $t \geq 1$.

For the structure of $A$, let $e_i = 1_{\mcD_i}$ for each $1 \leq i \leq t$. Then, $\{e_1,\ldots,e_t\}$ is the set of orthogonal idempotents related to the above decomposition of $B$. Each $e_i$ is a root of $X^2-X \in D[X]$, and so is in $A$ because $A=A'$. Thus, $A$ decomposes as $\prod_{i=1}^t A_i$, where $A_i=Ae_i\subset Be_i=\mcD_i$; in particular, $A_i$ is a (possibly noncommutative) domain for each $i$. Finally, assume $b_i' \in \mcD_i$ is integral over $D e_i \cong D$ with minimal polynomial $\mu_i(X) \in D[X]$. Then, the element $b = (b_1, \ldots, b_t) \in B$ with $b_i=b_i'$ and $b_j = 0$ for $j \ne i$ is a root of $X\mu_i(X)$, and hence is integral over $D$. Since $A=A'$, it follows that $b_i' \in A_i$, and so $A_i = A_i'$.
\end{proof}

For any ring $R$, let $\msJ(R)$ be the Jacobson radical of $R$ and let $\Max(R)$ be the set of maximal ideals of $R$. We recall that $R$ is \textit{semiprimitive} if $\msJ(R) = \{0\}$. If $D$ is semiprimitive (which occurs, for instance, when $D = \Z$), we can impose even stronger necessary conditions on $A$ in order for $\Int_K(A)$ to be \Pruf{}. To do this, we first show that in some cases the existence of a nonzero nilpotent element in a simple residue ring of $A$ is enough to force $\Int_K(A)$ to not be \Pruf{}.

\begin{Lem}\label{lem: a not in pA}
Assume that $D$ and $A$ satisfy both conditions (i) and (ii) below:
\begin{enumerate}[(i)]
\item $D$ contains a principal ideal $\pi D$ of finite index.
\item $A/\pi A$ has a simple residue ring isomorphic to $M_n(F)$, where $n \geq 2$ and $F$ is a finite field.
\end{enumerate}
Then, the following hold.
\begin{enumerate}[(1)]
\item There exists $a \in A$ such that $a^2 \in \pi^2 A$, but $a \notin \pi A$.
\item $\Int_K(A)$ is not \Pruf{}.
\end{enumerate}
\end{Lem}
\begin{proof}
(1) Let $R_1 = A/\pi A$ and $R_2 = A/\pi^2 A$, which have Jacobson radicals $J_1 := \msJ(R_1)$ and $J_2 := \msJ(R_2)$. By (i), $D/\pi D$ is finite, and $A$ is finitely generated as a $D$-module, so $R_1$ is finite. So, $R_2/J_2 \cong R_1/J_1$ is a finite semisimple ring. By (ii), $R_2/J_2$ contains a direct summand $S \cong M_n(F)$.

Write $1_S = \overline{e_1} + \cdots + \overline{e_n}$ for primitive orthogonal idempotents $\overline{e_1}, \ldots, \overline{e_n} \in S$. Then, $\overline{e_1} S \overline{e_2} \ne 0$ in $R_2/J_2$, and we may lift $\overline{e_1}$ and $\overline{e_2}$ to $R_2$ (see e.g.\ \cite[Prop.\ 21.25]{Lam1}). Thus, $R_2$ contains orthogonal idempotents $e_1$ and $e_2$ such that $e_1 R_2 e_2 \not\subseteq J_2$. Hence, there exists $b \in e_1 R_2 e_2$ such that $b \notin \pi R_2 \subseteq J_2$. However, $b^2 = 0$ because $(e_1 R_2 e_2)(e_1 R_2 e_2) = \{0\}$. It follows that the desired element $a$ exists in $A$.

(2) Let $a \in A$ satisfy the conclusion of (1). Then, $\Int_K(\{a\}, A)$ is not integrally closed, because $X^2/\pi^2 \in \Int_K(\{a\}, A)$, but $X/\pi \notin \Int_K(\{a\}, A)$. Since $\Int_K(A) \subseteq \Int_K(\{a\},A)$, this implies that $\Int_K(A)$ is not \Pruf{}.
\end{proof}

\begin{Lem}\label{lem: max spectrum}
Let $\mfm \in \Max(A)$ and let $P = \mfm \cap D$.
\begin{enumerate}[(1)]
\item $P$ is a nonzero prime ideal of $D$.
\item If $D$ is one-dimensional, then $P \in \Max(D)$.
\item If $\Int_K(A)$ is \Pruf{}, then $P \in \Max(D)$, $P$ has finite index in $D$, and $\mfm$ has finite index in $A$.
\end{enumerate}
\end{Lem}
\begin{proof}
(1) Certainly, $P$ is a proper ideal of $D$ because $1 \notin P$. Let $cd \in P$. Then, $cd \in \mfm$, and since both $c$ and $d$ are central in $A$, we have $c A d \subseteq \mfm$. The ring $A$ may be noncommutative, but maximal ideals in noncommutative rings are still prime \cite[p.\ 165]{Lam1}, so the condition $c A d \subseteq \mfm$ implies that either $c \in \mfm$ or $d \in \mfm$ \cite[Prop.\ 10.2]{Lam1}. Thus, either $c$ or $d$ is in $P$. 

Suppose that $P=\{0\}$ and let $R = A/\mfm$. Then, $D$ would be isomorphic to a subring of $R$, because
\begin{equation*}
D = D/(\mfm \cap D) \cong (D + \mfm)/\mfm \subseteq R.
\end{equation*}
Now, $R$ is a simple ring and $D$ is central in $R$, so for each nonzero $d \in D$ we have $dR = R$. So, $d^{-1}$ exists in $R$. Since $A$ is integral over $D$, this implies that $d^{-1}$ is integral over $D$. It follows that $K$ is integral over $D$. By \cite[Proposition 5.7]{AtiyahMacDonald}, $D$ is a field. This is a contradiction. Thus, $P \ne \{0\}$.

(2) This follows from (1).

(3) Assume that $\Int_K(A)$ is \Pruf{}. Then, $\Int(D)$ is also \Pruf{}. By \cite[Prop.\ 6.3]{Chabert1987}, $D$ must be an almost Dedekind domain with finite residue fields. By (2), $P \in \Max(D)$, and hence $D/P$ is finite. Since $A$ is finitely generated as a $D$-module, $A/\mfm$ is also finite.
\end{proof}

\begin{Thm}\label{thm: IntK(A) not Prufer}
Assume that $A$ contains a maximal ideal $\mfm$ such that $A/\mfm \cong M_n(F)$, where $n \geq 2$ and $F$ is a finite field. Then, $\Int_K(A)$ is not \Pruf{}.
\end{Thm}
\begin{proof}
Let $P = D \cap \mfm$. By Lemma \ref{lem: max spectrum}(1), $P$ is a nonzero prime ideal of $D$. Suppose that $\Int_K(A)$ is \Pruf{}. Then, as in Lemma \ref{lem: max spectrum}, $D$ is an almost Dedekind domain and $P$ has finite index in $D$. So, $PD_P$ is principal. Let $\pi \in D_P$ be a generator of $PD_P$. Then, $A_P/\pi A_P$ has $M_n(F)$ as a residue ring and $[D_P:\pi D_P] = [D: P]$ is finite. By Lemma \ref{lem: a not in pA}(1), there exists $\alpha \in A_P$ such that $\alpha^2 \in \pi^2 A_P$, but $\alpha \notin \pi A_P$. Let $\alpha = a/d$, where $a \in A$ and $d \in D \setminus P$. Then, the ring $\Int_K(\{a\}, A_P)$ is not integrally closed, because it contains $X^2/(d \pi)^2$ but not $X/(d \pi)$. So, $\Int_K(\{a\}, A_P)$ is not \Pruf{}, which is a contradiction because $\Int_K(A) \subseteq \Int_K(\{a\}, A_P)$.
\end{proof}

\begin{Cor}\label{cor: IntK(A) not Prufer v1}
If $\Int_K(A)$ is \Pruf{}, then every simple residue ring of $A$ is a finite field.
\end{Cor}
\begin{proof}
Apply Theorem \ref{thm: IntK(A) not Prufer} and Lemma \ref{lem: max spectrum}.
\end{proof}

When $D$ is semiprimitive, combining Theorem \ref{thm: A is a direct product} with Corollary \ref{cor: IntK(A) not Prufer v1} allows us to show that $A$ is a commutative ring whenever $\Int_K(A)$ is \Pruf{}. In that case, the structure of $A$ is further restricted by Theorem \ref{thm: loper}. Towards this end, we now study how $D$ being semiprimitive affects $A$.

\begin{Lem}\label{lem: extending maximal ideals}
Assume that $D$ is one-dimensional. Let $M \in \Max(D)$. Then, there exists a maximal ideal of $A$ containing $M A$.
\end{Lem}
\begin{proof}
Consider the localization $D_M$ and the $D_M$-algebra $A_M$. The ring $A_M$ cannot be simple, since then $d A_M =A_M$ for each $d \in D \setminus\{0\}$. This implies that $K = A_M \cap K$, and consequently $D \subsetneqq A \cap K$. So, $A_M$ contains a nonzero maximal ideal $\mfM$. The ideal $\mfM \cap D_M$ is maximal in $D_M$ by Lemma \ref{lem: max spectrum}(2), so we have $M D_M = \mfM \cap D_M$. It follows that the contraction of $\mfM$ to $A$ is a proper ideal of $A$ that contains $M$.
\end{proof}

\begin{Prop}\label{prop: when J(D)=0}
Assume that $B$ is a division ring.
\begin{enumerate}[(1)]
\item If $\mfI$ is a two-sided ideal of $A$ such that $\mfI \cap D = \{0\}$, then $\mfI = \{0\}$.

\item If $D$ is one-dimensional and semiprimitive, then the intersection of all the maximal ideals of $A$ is $\{0\}$. Consequently, $A$ is semiprimitive.
\end{enumerate}
\end{Prop}
\begin{proof}
(1) Assume that $\mfI$ is a two-sided ideal of $A$ with $\mfI \cap D = \{0\}$. Let $U = D\setminus\{0\}$. Since $B=U^{-1}A$ is a division ring, the extended ideal $U^{-1}\mfI$ of $B$ is equal to either $\{0\}$ or $B$. If $U^{-1}\mfI = B$, then $1=d^{-1} d$ for some nonzero $d \in \mfI \cap D$, which is a contradiction. Thus, $U^{-1}\mfI = \{0\}$, and this implies that $\mfI = \{0\}$.

(2) Assume that $D$ is semiprimitive and let $\mfI = \bigcap_{\mfm \in \Max(A)} \mfm$. Let $d \in \mfI \cap D$. Using Lemma \ref{lem: extending maximal ideals} and the fact that $D$ is commutative and semiprimitive,
\begin{equation*}
d \in \bigcap_{\mfm \in \Max(A)} (\mfm \cap D) = \bigcap_{M \in \Max(D)} M = \{0\}.
\end{equation*}
Thus, $\mfI \cap D = \{0\}$, and so $\mfI=\{0\}$ by (1). Finally, $A$ is semiprimitive because $\msJ(A) \subseteq \mfI$.
\end{proof}

\begin{Prop}\label{prop: when A is a direct product}
Assume that $B = \prod_{i=1}^t \mcD_i$, where each $\mcD_i$ is a division ring. For each $1 \leq i \leq t$, let $e_i = 1_{\mcD_i}$ and let $D_i = De_i \cong D$. Assume also that $A = \prod_{i=1}^t A_i$, where each $A_i$ is a $D_i$-subalgebra of $\mcD_i$.  If $D$ is one-dimensional and semiprimitive, then the intersection of all the maximal ideals of $A$ is $\{0\}$. Consequently, $A$ is semiprimitive.
\end{Prop}
\begin{proof}
Let $\mfI = \bigcap_{\mfm \in \Max(A)} \mfm$. Since $A$ is a direct product of rings, we have $\mfI = \prod_{i=1}^t \mfI_i$, where each $\mfI_i$ is a two-sided ideal of $A_i$. We claim that 
\begin{equation*}\label{eq: max id intersection}
\mfI_i = \bigcap_{\mfm_i \in \Max(A_i)} \mfm_i
\end{equation*}
for each $i$. To see this, note that each maximal ideal $\mfm \subseteq A$ has the form $\mfm = \prod_{i=1}^t \mfa_i$, where $\mfa_j$ is a maximal ideal of $A_j$ for exactly one $j \in \{1, \ldots, t\}$, and $\mfa_i = A_i$ for $i \ne j$. Intersecting all of the maximal ideals of $A$ yields the desired decomposition of $\mfI$.

Now, if $D$ is semiprimitive, then so is each $D_i$. Hence, $\mfI_i = \{0\}$ for each $1 \leq i \leq t$ by Lemma \ref{prop: when J(D)=0}(2), and therefore $\mfI = \{0\}$. As in Proposition \ref{prop: when J(D)=0}, $A$ is semiprimitive because $\msJ(A) \subseteq \mfI$.
\end{proof}

\begin{Thm}\label{thm: D semiprimitive}
If $\Int_K(A)$ is \Pruf{} and $D$ is semiprimitive, then $A \cong \prod_{i=1}^t A_i$, where $A_i$ is the integral closure of $D$ in a finite field extension $F_i/K$. In particular, each $A_i$ is an \OurD-domain and $A$ is commutative.
\end{Thm}
\begin{proof}
Assume $\Int_K(A)$ is \Pruf{}. First, we prove that $A$ is a commutative ring. Let $\mfI$ be the intersection of all the maximal ideals of $A$. Since $\Int_K(A)$ is \Pruf{} and $D$ is semiprimitive, by Theorem \ref{thm: A is a direct product} and Proposition \ref{prop: when A is a direct product}, $\mfI = \{0\}$. Moreover, by Corollary \ref{cor: IntK(A) not Prufer v1}, $A/\mfm$ is a finite field for each $\mfm \in \Max(A)$. So, for all $a, b \in A$ and all $\mfm \in \Max(A)$, we have $ab-ba \in \mfm$. Hence, $ab-ba \in \mfI$, and therefore $ab=ba$.

Next, apply Theorem \ref{thm: A is a direct product} and express $B$ and $A$ as $B \cong \prod_{i=1}^t \mcD_i$ and $A \cong \prod_{i=1}^t A_i$, where each $\mcD_i$ is a division ring and each $A_i$ is an integrally closed domain. For each $1 \leq i \leq t$, let $e_i = 1_{\mcD_i}$, $D_i = De_i \cong D$, and $K_i = Ke_i \cong K$. Then, $\mcD_i = (D_i \setminus \{0\})^{-1} A_i$ for each $i$. Since $A$ is commutative, so is each $A_i$. Thus, each $\mcD_i$ is a field $F_i$, and $[F_i:K_i]$ is finite because $\dim_K(B)$ is finite. 

Since $\Int(D)$ is \Pruf{}, by Theorem \ref{thm: loper}, $D$---and hence each $D_i$---is an \OurD-domain. From Theorem \ref{thm: A is a direct product}, we know that $A=A'$ and $A_i=A_i'$ for each $i$. This means that $A_i$ is the integral closure of $D_i$ in $F_i$, and is itself an \OurD{}-domain by \cite[Theorem 36.1(a)]{Gilm}.
\end{proof}

If $D$ is not semiprimitive, then there exist examples of noncommutative rings $A$ such that $\Int_K(A)$ is \Pruf{}; see Theorem \ref{thm: Z2 Hurwitz}. However, the following partial converse of Theorem \ref{thm: D semiprimitive} is true.

\begin{Thm}\label{thm: commutative converse}
Let $D$ be an \OurD-domain. Let $A \cong \prod_{i=1}^t A_i$, where $A_i$ is the integral closure of $D$ in a finite field extension $F_i/K$. Then $\Int_K(A)$ is \Pruf{}. 
\end{Thm}
\begin{proof}
Note that each $A_i$ is an \OurD-domain \cite[Theorem 36.1(a)]{Gilm}. Let $F$ be a finite Galois extension of $K$ containing $F_i$ for $i=1,\ldots, t$ and let $D_F$ be the integral closure of $D$ in $F$. Note that $D_F$ is also an \OurD-domain. So, $\Int(D_F)$ is \Pruf{} by Theorem \ref{thm: loper}.

Let $G=\Gal(F/K)$, let $i\in\{1,\ldots,t\}$, and let $G(A_i)= \bigcup_{\sigma \in G} \sigma(A_i)$. It is easy to show that $G(A_i)=\Omega_{A_i}$ and $\Int_K(A_i) = \Int_K(G(A_i),\oD) = \Int_K(G(A_i),D_F)$, and this last ring is integrally closed by Theorem \ref{thm: Int_K(A,A')}. Hence, we have
\begin{equation*}
\Int_K(A) = \bigcap_{i=1}^n \Int_K(A_i) = \bigcap_{i=1}^n \Int_K(G(A_i), D_F) = \Int_K(\bigcup_{i=1}^n G(A_i), D_F).
\end{equation*}
Now, by \cite[Lemma 2.20]{PerDedekind}, $\Int_F(\bigcup_{i=1}^n G(A_i), D_F)$ is the integral closure of $\Int_K(\bigcup_{i=1}^n G(A_i), D_F)$ in $F(X)$. Since $\Int(D_F) \subseteq \Int_F(\bigcup_{i=1}^n G(A_i), D_F)$ and the former ring is a \Pruf{} domain, it follows that $\Int_F(\bigcup_{i=1}^n G(A_i), D_F)$ is also \Pruf{}. By \cite[Theorem 22.4]{Gilm}, $\Int_K(\bigcup_{i=1}^n G(A_i), D_F)$ is \Pruf{}, as we wanted to prove.  
\end{proof}

\begin{Cor}\label{cor: D semiprimitive}
Let $D$ be a semiprimitive \OurD-domain. Then, $\Int_K(A)$ is \Pruf{} if and only if $A \cong \prod_{i=1}^t A_i$, where $A_i$ is the integral closure of $D$ in a finite field extension $F_i/K$.
\end{Cor}

\section{Sufficient conditions}\label{sec: sufficient}

In this section, we complete the proof of Theorem \ref{thm: main} by showing that if $D$ is an \OurD-domain and $A=A'$, then $\Int_K(A)$ is \Pruf{}. This is accomplished by constructing a family of \Pruf{} domains within $K[X]$, one of which must be a subring of $\Int_K(A,A') := \{f \in K[X] \mid f(A) \subseteq A'\}$ (which is equal to $\Int_K(A)$ if $A=A'$). To construct these \Pruf{} domains, we generalize \cite[Section 2]{LopWer}, which dealt with a similar problem for $D = \Z$.

Recall from the introduction that $\oK$ denotes a fixed algebraic closure of $K$, and from Definition \ref{def: Int_K(A,A')} that for each $n \geq 1$, $\Lambda_n$ is the set of elements of $\oK$ that are integral over $D$ of degree at most $n$. For $\msE \subseteq \Lambda_n$, define $\Int_K(\msE, \Lambda_n) = \{f \in K[X] \mid f(\msE) \subseteq \Lambda_n\}$. As usual, we let $\Int_K(\Lambda_n) := \Int_K(\Lambda_n, \Lambda_n)$.

It is shown in \cite[Prop.\ 7]{PerWer} that for any $\msE \subseteq \Lambda_n$, $\Int_K(\msE, \Lambda_n)$ is an integrally closed subring of $K[X]$. We wish to prove that $\Int_K(\Lambda_n)$ is a \Pruf{} domain for each $n \geq 1$. To do this, we provide a broad construction that produces a \Pruf{} domain within $K[X]$.

\subsection{Building a Pr\"{u}fer domain inside $K[X]$}

\begin{Constr}\label{Building D}
Let $D$ be an almost Dedekind domain with finite residue fields. As mentioned in Notation \ref{D_0 notations}, $D$ contains a principal ideal domain $D_0$ and we use $\mfp$ to denote a nonzero prime element of $D_0$. 

For each nonzero prime $P$ of $D$, let $I_P$ be an index set and let $\{V_{P,i} \mid i \in I_P\}$ be a collection of valuation overrings of $D[X]$ such that for each $i \in I_P$, the valuation $v_{P,i}$ of $K(X)$ extends $v_P$, and $V_{P,i}$ has maximal ideal $M_{P,i}$. Let $\mfp$ be the prime of $D_0$ below $P$. Assume that the following two conditions hold:
\begin{enumerate}[(i)]
\item there exists $e_\mfp \geq 1$ such that for all $P|\mfp D_0$ and all $i \in I_P$, $M_{P,i}^{e_\mfp} \subseteq \mfp V_{P,i}$, and
\item there exists $f_\mfp \geq 1$ such that for all $P|\mfp D_0$ and all $i \in I_P$, $[(V_{P,i}/M_{P,i}) : (D_0/\mfp D_0)] \leq f_\mfp$.
\end{enumerate}
Let $\Spec(D)$ and $\Max(D)$ be, respectively, the sets of prime and maximal ideals of $D$. We then form the following subring of $K[X]$:
\begin{equation*}
\msD := \bigcap_{P\in\Max(D)}  \Big( \; \bigcap_{i \in I_P} V_{P,i} \Big) \cap K[X].
\end{equation*}
\end{Constr}

\begin{Rem}\label{rem: M_i}
Since $D$ is assumed to be an almost Dedekind domain, each valuation $v_P$ is discrete. The existence of $e_\mfp$ implies that for each $P|\mfp D_0$ and each $i \in I_P$, there exists a positive integer $\lambda_{P,i}\leq e_{\mfp}$ such that $M_{P,i}^{\lambda_{P,i}} = \mfp V_{P,i}$. Indeed, the containments $M_{P,i}^{e_\mfp} \subseteq \mfp V_{P,i}\subseteq M_{P,i}$ imply that $\mfp V_{P,i}$ is $M_{P,i}$-primary, so by \cite[Theorem 17.3(b)]{Gilm} $\mfp V_{P,i}$ is equal to a power $M_{P,i}^{\lambda_{P,i}}$ of the maximal ideal $M_{P,i}$ and necessarily $\lambda_{P,i}\leq e_{\mfp}$. 
\end{Rem}

For the remainder of this section, assume that $\msD$ is a domain formed via Construction \ref{Building D}. We will prove that $\msD_\mfM$ is a valuation domain for each maximal ideal $\mfM$ of $\msD$, and hence that $\msD$ is a Pr\"{u}fer domain. The ideals of $\msD$ fall into two types. An ideal $\mfI \subseteq \msD$ is \textit{unitary} if $\mfI \cap D \ne \{0\}$, and is \textit{non-unitary} otherwise. The non-unitary maximal ideals of $\msD$ can be dealt with quite quickly.

\begin{Lem}\label{lem: non-unitary}
If $\mfM$ is a non-unitary maximal ideal of $\msD$, then $\msD_\mfM$ is a valuation domain.
\end{Lem}
\begin{proof}
When $\mfM$ is non-unitary and maximal, we have $K[X] \subseteq \msD_\mfM$. Thus, $\msD_\mfM$ is a local overring of the Pr\"{u}fer domain $K[X]$, and hence is a valuation domain.
\end{proof}

Working with the unitary maximal ideals of $\msD$ is more complicated. When $\mfM$ is unitary, its contraction to $D$ is a nonzero prime $P \subseteq D$, and $P$ itself lies above a prime element $\mfp$ of $D_0$.

\begin{Not}\label{r_p and r_p}
Recall the definitions of the positive integers $e_\mfp$ and $f_\mfp$ from Construction \ref{Building D}. For each $\mfp$, let $s_\mfp = e_\mfp !$ and $r_\mfp = |D_0 / \mfp D_0|^{f_\mfp !}$. Note that $r_\mfp \geq 2$.
\end{Not}

\begin{Prop}\label{prop: polynomials in D and  M}\mbox{}
\begin{enumerate}[(1)]
\item For each $\mfp$ and each $f \in \msD$, $\dfrac{(f^{r_\mfp} - f)^{s_\mfp}}{\mfp} \in \msD$.
\item Let $\mfM$ be a unitary prime ideal of $\msD$ above $\mfp$. Then, $f^{r_\mfp} - f \in \mfM$ for all $f \in \msD$.
\end{enumerate}
\end{Prop}
\begin{proof}
For readability, let $r = r_\mfp$ and $s = s_\mfp$.

(1) First, we show that $(f^r -f)^s \in \mfp V_{P,i}$ for each $P|\mfp D_0$ and each $i \in I_P$. By construction, $[(V_{P,i}/M_{P,i}) : (D_0 / \mfp D_0)]$ divides $f_\mfp !$ and $e_\mfp \leq s$. So, $f^r - f \in M_{P,i}$, and then $(f^r -f)^s \in M_{P,i}^s \subseteq \mfp V_{P,i}$. Thus, $(f^r-f)^s/\mfp \in V_{P,i}$.

Next, if $P'$ is a nonzero prime of $D$ that is not above $\mfp$, then $\mfp$ is invertible in each valuation overring $V'$ of $D$ associated to $P'$. In this case, $(f^r-f)^s/\mfp \in V'$. We conclude that $(f^r-f)^s/\mfp \in \msD$.

(2) By part (1), $(f^r - f)^s \in \mfp \msD \subseteq \mfM$ for each $f \in \msD$. Consequently, $f^r -f \in \mfM$.
\end{proof}

\begin{Lem}\label{lem: might be false}
Let $f \in \msD$. For every $\mfp$, every $P|\mfp D_0$, and each $i \in I_P$, one of the following two conditions holds. Either
\begin{enumerate}[(i)]
\item $v_{P,i}\Big(\dfrac{f^{s_\mfp}}{\mfp^t}\Big) = 0$ for some $t \geq 0$, or
\item $v_{P,i}\Big(\dfrac{f^{s_\mfp}}{\mfp^t}\Big) > 0$ for all $t \geq 0$.
\end{enumerate}
\end{Lem}
\begin{proof}
Fix $\mfp$, $P$, and $i$. Let $V=V_{P,i}$, $M=M_{P,i}$, $v=v_{P,i}$, and $s=s_\mfp$. Since $f \in \msD$, $v(f) \geq 0$. If $v(f) = 0$, then condition (i) holds with $t=0$. So, assume that $v(f) > 0$. This means that $f \in M$. If condition (ii) holds, then we are done. So, assume not. Then, there exists some $n\geq 1$ such that $\mfp^n V\subseteq f^s V$.   By the same argument of Remark \ref{rem: M_i}, $fV=M^a$, for some $a\geq1$, so $f^sV=M^{sa}$.  Recall from Remark \ref{rem: M_i} that $\mfp V = M^{\lambda}$ for some integer $\lambda \geq 1$. Note that $\lambda$ divides $s$ because $\lambda\leq e_{\mfp}$. Hence, there exists some $t\geq 1$ such that $sa=\lambda t$, so that $f^sV=M^{\lambda t}=\mfp^t V$. Hence, $v(f^s)=v(\mfp^t)$, which is (i).  
\end{proof}

\begin{Lem}\label{lem: new poly in D}
Let $f \in \msD$. Assume that for some $\mfp$, every $P|\mfp D_0$, and every $i \in I_P$, $f$ satisfies one of the following two conditions. Either
\begin{enumerate}[(i)]
\item $v_{P,i}\Big(\dfrac{f}{\mfp^t}\Big) = 0$ for some $t \geq 0$, or
\item $v_{P,i}\Big(\dfrac{f}{\mfp^t}\Big) > 0$ for all $t \geq 0$.
\end{enumerate}
Let $g = \dfrac{f(f^{r_{\mfp}-1} - 1)^{s_\mfp}}{\mfp} \in K[X]$. Then, $g$ also satisfies (i) or (ii), and $g \in \msD$.
\end{Lem}
\begin{proof}
As before, let $r = r_\mfp$ and $s = s_\mfp$. First we show that $g$ satisfies either (i) or (ii) for each $P|\mfp D_0$ and each $i \in I_P$. Fix $P$ and $i \in I_P$, let $V=V_{P,i}$, $M=M_{P,i}$, and $v=v_{P,i}$. By Proposition \ref{prop: polynomials in D and  M}(2), $f(f^{r-1}-1) = f^r -f \in \mfM$ for each unitary maximal ideal of $\msD$ above $\mfp$. So, either $v(f) > 0$ or $v(f^{r-1}-1) > 0$. Note that these two conditions cannot occur simultaneously.

Assume first that $v(f) > 0$. Then, $v(g) = v(\tfrac{f}{\mfp})$. If (i) holds for $f$, then $v(f) = tv(\mfp)$ for some integer $t \geq 1$. So, $v(g) = (t-1)v(\mfp)$ and (i) holds for $g$. On the other hand, if $f$ satisfies (ii), then so does $g$, because for all $t \geq 0$ we have $v(\tfrac{g}{\mfp^t}) = v(\tfrac{f}{\mfp^{t+1}}) > 0$.

Next, assume that $v(f^{r-1}-1) > 0$. Then, $f^{r-1}-1 \in M$, $(f^{r-1}-1)^s \in \mfp V$ and $v(f) = 0$ (so, in this case, condition (i) is true for $f$ with $t=0$). By Lemma \ref{lem: might be false}, either condition (i) or condition (ii) holds for $(f^{r-1}-1)^s$. Moreover, if (i) is true for $(f^{r-1}-1)^s$, then the fact that $(f^{r-1}-1)^s \in \mfp V$ means $v((f^{r-1}-1)^s/\mfp^t) = 0$ for some $t \geq 1$. We have $v(g) = v((f^{r-1}-1)^s/\mfp)$, so either (i) or (ii) holds for $g$.

Finally, the fact that $g$ satisfies either (i) or (ii) implies that $g \in V$, and this is true for every $P|\mfp D_0$ and every $i \in I_P$. This is enough to conclude that $g \in \msD$, because if $P'$ is a nonzero prime of $D$ that is not over $\mfp$, then $\mfp$ is a unit in each valuation overring associated to $P'$.
\end{proof}

\begin{Lem}\label{lem: final lemma}
Let $\mfM$ be a unitary prime ideal of $\msD$ above $\mfp$ and let $f \in \msD\setminus\{0\}$. Then, either $f^{s_\mfp}/\mfp^t \in \mfM\msD_\mfM$ for all integers $t \geq 1$, or $f^{s_\mfp}/\mfp^t$ is a unit of $\msD_\mfM$ for some integer $t \geq 0$.
\end{Lem}
\begin{proof}
Once again, let $r = r_\mfp$ and $s = s_\mfp$. There is nothing to prove if $f \notin \mfM$, so assume that $f \in \mfM$. Let $f_0 = f^s$, and for each $k \geq 1$, let
\begin{equation*}
f_k = \dfrac{f_{k-1}(f_{k-1}^{r-1}-1)^s}{\mfp}.
\end{equation*}
By Lemma \ref{lem: might be false}, $f_0$ meets the conditions of Lemma \ref{lem: new poly in D}. Thus, $f_1 \in \msD$, and by induction $f_k \in \msD$ for all $k \geq 0$.

For each $k \geq 1$, let $h_k = \prod_{i=0}^{k-1} (f_i^{r-1}-1)^s$. Then, $f_k = (f^s h_k)/\mfp^k$ for all $k \geq 1$. Notice that if $f_i \in \mfM$ for all $0 \leq i \leq k-1$, then $h_k \notin \mfM$. Hence, if there exists a smallest positive integer $t$ such that $f_t \notin \mfM$, then $f^s/\mfp^t = f_t/h_t$ is a unit in $\msD_\mfM$. Otherwise, $f_t \in \mfM$ for every $t \geq 1$, each $h_t \notin \mfM$, and $f^s/\mfp^t = f_t/h_t \in \mfM \msD_\mfM$ for all $t \geq 1$.
\end{proof}

\begin{Thm}\label{thm: D is Prufer}
Let $\msD$ be a domain formed via Construction \ref{Building D}. Then, $\msD$ is Pr\"{u}fer.
\end{Thm}
\begin{proof}
Let $\mfM$ be a maximal ideal of $\msD$. If $\mfM$ is non-unitary, then $\msD_\mfM$ is a valuation domain by Lemma \ref{lem: non-unitary}. So, assume that $\mfM$ is a unitary maximal ideal above $\mfp$. Let $f, g \in \msD$ be nonzero and coprime in $K[X]$. We will show that either $f/g$ or $g/f$ is in $\msD_\mfM$. This is clear if either $f \notin \mfM$ or $g \notin \mfM$. So, assume that both $f$ and $g$ are in $\mfM$.

Let $s = s_\mfp$. Since $f$ and $g$ are coprime in $K[X]$, there exist $h_1, h_2 \in K[X]$ such that $1 = h_1 f^s + h_2 g^s$. Choose $d \in D \setminus\{0\}$ such that $dh_1, dh_2 \in D[X] \subseteq \msD$. If both $f^s/\mfp^t$ and $g^s/\mfp^t$ are in $\mfM \msD_\mfM$ for every positive integer $t$, then
\begin{equation*}
\dfrac{d}{\mfp^t} = \dfrac{dh_1 f^s + dh_2 g^s}{\mfp^t} \in \mfM \msD_\mfM
\end{equation*}
for all $t \geq 1$. This implies that $d/\mfp^t \in P = \mfM \cap D$ for all $t \geq 1$, which is impossible. By Lemma \ref{lem: final lemma}, we can find $t \geq 1$ such that both $f^s/\mfp^t$ and $g^s/\mfp^t$ are in $\msD_\mfM$, and either $f^s/\mfp^t$ or $g^s/\mfp^t$ is a unit of $\msD_\mfM$. Without loss of generality, assume that $f^s/\mfp^t$ is a unit of $\msD_\mfM$. Then,
\begin{equation}\label{g/f eq}
(g/f)^s = g^s/f^s = (g^s/\mfp^t) / (f^s/\mfp^t) \in \msD_\mfM.
\end{equation}
Since $\msD$ is an intersection of valuation domains, it is integrally closed. So, $\msD_\mfM$ is also integrally closed. Thus, $g/f \in \msD_\mfM$, $\msD_\mfM$ is a valuation domain, and $\msD$ is Pr\"{u}fer.
\end{proof}

\subsection{Proving that $\Int_K(\Lambda_n)$ is Prufer}

Next, we show how to express $\Int_K(\Lambda_n)$ as the intersection of a collection of valuation domains like those of Construction \ref{Building D}. It then follows from Theorem \ref{thm: D is Prufer} that $\Int_K(\Lambda_n)$ is \Pruf{}, and this will be enough to complete the proof of Theorem \ref{thm: main}(1).

For each $P\in\Max(D)$, we let $\widehat{D_P}$ and $\widehat{K}_P$ be the $P$-adic completions of $D$ and $K$, respectively. Let $\overline{\widehat{K}_P}$ be a fixed algebraic closure of $\widehat{K}_P$. Since $D_P$ is a rank one valuation domain, it follows that $v_P$ extends uniquely to $\overline{\widehat{K}_P}$ and the valuation domain of this unique extension is equal to $\overline{\widehat{D_P}}$, the absolute integral closure of $\widehat{D_P}$. We still denote this unique extension by $v_P$.  For $n\geq 1$, we set $\widehat{\Lambda}_{P,n}:=\{\alpha\in \overline{\widehat{D_P}}\mid [\widehat{K}_P(\alpha):\widehat{K}_P]\leq n\}$. Finally, for $\alpha\in\overline{\widehat{D_P}}$, we define $V_{P,\alpha}:=\{\phi\in K(X)\mid \phi(\alpha)\in \overline{\widehat{D_P}}\}$, which is a valuation domain of $K(X)$ lying above $D_P$ by \cite[Proposition 2.2]{PerTransc}.
\begin{Thm}\label{thm: Lambda_n}
Let $n\geq 1$. Then,
$$\Int_K(\Lambda_n)=\bigcap_{P\in\Max(D)}\bigcap_{\alpha\in\widehat{\Lambda}_{P,n}}V_{P,\alpha}\cap K[X].$$
\end{Thm}
\begin{proof}
Let $f\in\Int_K(\Lambda_n)$ and write $f(X)=\frac{g(X)}{d}$ for some $g\in D[X]$ and $d\in D\setminus\{0\}.$ For each $P\in\Max(D)$ such that $d\notin P$, we have $f\in D_P[X]$, so $f\in V_{P,\alpha}$ for each $\alpha\in\widehat{\Lambda}_{P,n}$. Next, let $P\in\Max(D)$ be such that $d\in P$ and fix $\alpha\in \widehat{\Lambda}_{P,n}$. Let $k=v_P(d)\in\Z$. Let $p_{\alpha} \in \widehat{D_P}[X]$ be the minimal polynomial of $\alpha$. 

Extend the valuation $v_P$ to $\widehat{K}_P[X]$ by taking $v_P(h)$ to be the minimum of the values of the coefficients of $h \in \widehat{K}_P[X]$. Since $D$ is dense in $\widehat{D_P}$, for any $\delta \in \Z$ we can find a monic $q \in D[X]$ such that $\deg(q) = \deg(p_{\alpha})$ and $v_P(p_{\alpha} - q) > \delta$. Let $\varepsilon \in \Z$ such that $\varepsilon > \max\{k,v_P(\alpha_i-\alpha_j)\mid \alpha_i\not=\alpha_j\in\Omega_{\alpha}\}$. Then, by \cite[Chapter 1, (P)]{Rib}, there exists $q \in D[X]$ monic of degree $\deg(p_\alpha)$ and a root $\gamma$ of $q$ in $\overline{\widehat{D_P}}$ such that $v_P(\alpha - \gamma) \geq \varepsilon > k$. 

Next, for each root $\lambda \in \Omega_q$ of $q$ in $\overline{K}$, let $D_\lambda'$ be the integral closure of $D$ in $K(\lambda)$. By \cite[Chapter 4, (L)]{Rib}, for each $\lambda$ and for each prime $Q$ of $D_\lambda'$ above $P$, there exists a $K$-embedding $\sigma_{\lambda,Q}: K(\lambda) \to \overline{\widehat{K}_P}$ such that the extension $v_Q$ of $v_P$ to $K(\lambda)$ is given by $v_Q = v_P \circ \sigma_{\lambda,Q}$.  Choose $\beta \in \Omega_q$ with $K$-embedding $\sigma:K(\beta)\to\overline{\widehat{K}_P}$ such that $\sigma(\beta)$ is equal to the root $\gamma$ of $q$ in $\overline{\widehat{D_P}}$ chosen as above. Then, by the previous paragraph, $v_P(\alpha - \sigma(\beta)) > k$. Since $g$ has coefficients in $D$, $\alpha - \sigma(\beta)$ divides $g(\alpha) - g(\sigma(\beta))$ in $\overline{\widehat{D_P}}$. Thus, $v_P(g(\alpha) - g(\sigma(\beta))) > k$. Furthermore, $v_P(g(\sigma(\beta))) = v_P(\sigma(g(\beta))) = v_Q(g(\beta))$, and $v_Q(g(\beta)) \geq k$ because $f \in \Int_K(\Lambda_n)$ and $\beta \in \Lambda_n$. It follows that $v_P(g(\alpha)) \geq k$. Hence, $f(\alpha)=\frac{g(\alpha)}{d} \in \overline{\widehat{D_P}}$ and $f \in V_{P,\alpha}$.

Conversely, let $f\in K[X]$ be in each of the valuation domains $V_{P,\alpha}$, for $P\in\Max(D)$ and $\alpha\in \widehat{\Lambda}_{P,n}$. Let $\beta\in\Lambda_n$ and suppose that $f(X)=\frac{g(X)}{d}$ for some $g\in D[X]$ and $d\in D\setminus\{0\}.$ We set $F=K(\beta)$ and let $D'$ be the integral closure of $D$ in $F$. Let $P\in\Max(D)$.  If $d\not\in P$, then $f\in D_P[X]$. In this case, $f(\beta)\in (D')_{Q}$ for each prime ideal $Q$ of $D'$ lying above $P$, so $f(\beta)\in (D')_P$, which is the integral closure of $D_P$ in $F$. Suppose now that $d\in P$. Fix a prime ideal $Q$ in $D'$ lying above $P$. By \cite[Chapter 4, (L)]{Rib}, there is an embedding $\sigma_Q : F \to \overline{\widehat{K}_P}$ such that the extension $v_Q$ of $v_P$ to $F$ is given by $v_Q = v_P \circ \sigma_Q$. Let $\alpha = \sigma_Q(\beta)$. Since $f \in V_{P,\alpha}$, we have $v_P(f(\alpha)) \geq 0$. So, $v_Q(f(\beta)) \geq 0$. This holds for each $Q$ above $P$, so $f(\beta) \in (D')_P$. We have shown that $f(\beta)$ is in $(D')_P$ for each $P \in \Max(D)$. It follows that $f(\beta) \in D'$ and $f \in \Int_K(\Lambda_n)$.
\end{proof}

\begin{Cor}\label{cor: Lambda_n is Prufer}
Let $D$ be an \OurD-domain.
\begin{enumerate}[(1)]
\item For each $n \geq 1$, $\Int_K(\Lambda_n)$ is \Pruf{}.
\item Let $A$ be a $D$-algebra that is torsion-free and finitely generated as a $D$-module, and such that $A \cap K = D$. Then, $\Int_K(A,A')$ is \Pruf{}.
\end{enumerate}
\end{Cor}
\begin{proof}
(1) Let $n \geq 1$ and express $\Int_K(\Lambda_n)$ as in Theorem \ref{thm: Lambda_n}. Let $\mfp$ be a nonzero prime element of $D_0$. Since $D$ is an \OurD-domain, we know that the integers $E_{\mfp}$ and $F_{\mfp}$ from Notation \ref{D_0 notations} are bounded above, say by some integer $d_{\mfp}$. Let $P$ be a prime of $D$ above $\mfp$. By the results of \cite[Proposition 2.2]{PerTransc}, we have  $M_{P,\alpha}^{e_{P,\alpha}}=PV_{P,\alpha}$ for some $e_{P,\alpha}\leq n$ and $[V_{P,\alpha}/M_{P,\alpha}:D/P]\leq n$  for each $\alpha\in \widehat{\Lambda}_{P,n}$. Therefore, $\Int_K(\Lambda_n)$ is a domain formed by Construction \ref{Building D}, because we can take $e_{\mfp}=f_{\mfp}=nd_\mfp$. Hence, by Theorem \ref{thm: D is Prufer}, $\Int_K(\Lambda_n)$ is a Pr\"ufer domain.

(2) Let $n = \dim_K(B)$. Then, $\Int_K(A,A') = \Int_K(\Omega_A, \Lambda_n)$ by Theorem \ref{thm: Int_K(A,A')}. Since $\Omega_A \subseteq \Lambda_n$, we have $\Int_K(\Lambda_n) \subseteq \Int_K(\Omega_A, \Lambda_n)$. By part (1), $\Int_K(\Lambda_n)$ is \Pruf{}. Hence, $\Int_K(A,A')$ is also \Pruf{}.
\end{proof}

\subsection{The proof of Theorem \ref{thm: main}}

To finish the proof of Theorem \ref{thm: main}, it remains to consider the conditions that involve localizations at primes of $D$. For a prime $P \subseteq D$, let $(A')_P = (D\setminus P)^{-1}A'$ and $(A_P)' = \{b \in B \mid b \text{ is integral over } D_P\}$. It is straightforward to check that $(A')_P = (A_P)'$, so we may take $A_P' := (A')_P = (A_P)'$ and use the notation $A_P'$ without ambiguity.

\begin{Lem}\label{lem: Rush trick}
Let $P \in \Spec(D)$. Then, $\Int_K(A)_P \subseteq \Int_K(A_P)$ and $\Int_K(A,A')_P \subseteq \Int_K(A_P, A_P')$.
\end{Lem}
\begin{proof}
We prove that $\Int_K(A,A')_P \subseteq \Int_K(A_P, A_P')$ by employing an induction technique due to Rush \cite[Proposition 1.4]{Rush}. The proof that $\Int_K(A)_P \subseteq \Int_K(A_P)$ is similar. Let $f \in \Int_K(A,A')_P$ have degree $m$.  We use induction on $m$.   If $m=0$, then $f \in A_P'$ and we are done. Assume that $m \geq 1$ and that any polynomial in $\Int_K(A,A')_P$ of degree at most $m-1$ is an element of $\Int_K(A_P, A_P')$. Let $a/d \in A_P$, where $a \in A$ and $d \in D \setminus P$. Let $g(x) = d^m f(x) - f(dx)$. Then, both $f(dx)$ and $g$ are in $\Int_K(A,A')_P$, and $\deg g \leq m-1$. By induction, $g \in \Int_K(A_P, A_P')$.

Now, let $R = A_P' \cap K[a]$. It may the case that $A_P'$ is not a ring; however, $R$ is equal to the integral closure of $D_P$ in $K[a]$, and hence is a ring. Both $g(a/d)$ and $f(a)$ are in $R$, and $d$ is invertible in $R$, so $f(a/d) = d^{-m}(g(a/d) + f(a)) \in R \subseteq A_P'$. This is true for all $a/d \in A_P$, so $f \in \Int_K(A_P, A_P')$.
\end{proof}

\begin{Lem}\label{lem: Prufer local property}
Let $D$ be an \OurD{}-domain. The following are equivalent.
\begin{enumerate}[(i)]
\item $\Int_K(A)$ is \Pruf{}.
\item For each $P \in \Spec(D)$, $\Int_K(A)_P$ is \Pruf{}.
\item For each $P \in \Spec(D)$, $\Int_K(A_P)$ is \Pruf{}.
\item For each $P \in \Spec(D)$, $\Int_K(A_P)$ is integrally closed.
\item $\Int_K(A) = \Int_K(A,A')$.
\end{enumerate}
\end{Lem}
\begin{proof}
(i) $\Rightarrow$ (ii) $\Rightarrow$ (iii) Certainly, $\Int_K(A) \subseteq \Int_K(A)_P$ for all $P$, and $\Int_K(A)_P \subseteq \Int_K(A_P)$ by Lemma \ref{lem: Rush trick}.

(iii) $\Rightarrow$ (iv) This is clear, because \Pruf{} domains are integrally closed.

(iv) $\Rightarrow$ (v) Assume that $\Int_K(A_P)$ is integrally closed for every $P \in \Spec(D)$. On the one hand, we always have $\Int_K(A) \subseteq \Int_K(A,A')$. On the other hand,
\begin{equation*}%\label{eq: first containment}
\Int_K(A,A') \subseteq \bigcap_{P \in \Spec(D)} \Int_K(A,A')_P \subseteq \bigcap_{P \in \Spec(D)} \Int_K(A_P,A_P'),
\end{equation*}
where the second containment is true by Lemma \ref{lem: Rush trick}. By assumption, each $\Int_K(A_P)$ is integrally closed. Since $D$ is an \OurD{}-domain, $D_P$ is a DVR with finite residue field. So, $\Int_K(A_P) = \Int_K(A_P,A_P')$ by Theorem \ref{thm: Int_K(A,A')}. This gives
\begin{equation*}
\Int_K(A,A') \subseteq \bigcap_{P \in \Spec(D)} \Int_K(A_P,A_P') = \bigcap_{P \in \Spec(D)} \Int_K(A_P) \subseteq \Int_K(A).
\end{equation*}
Thus, $\Int_K(A) = \Int_K(A,A')$.

(v) $\Rightarrow$ (i) Apply Corollary \ref{cor: Lambda_n is Prufer}(2).
\end{proof}

We now have all the pieces necessary to complete the proof of Theorem \ref{thm: main}.

\begin{myproof}[of Theorem \ref{thm: main}]\mbox{}
Part (2) is Corollary \ref{cor: D semiprimitive}. For part (1), the conditions can be shown to be equivalent as follows.

(i) $\Rightarrow$ (v) This is Theorem \ref{thm: A is a direct product}.

(v) $\Rightarrow$ (iii) By assumption, $A_i=A_i'$ for each $i$, so $A=A'$.

(iii) $\Rightarrow$ (ii) This is clear.

(ii) $\Rightarrow$ (i) Use Corollary \ref{cor: Lambda_n is Prufer}(2).

(iii) $\Leftrightarrow$ (iv) This is Lemma \ref{pointwise integrally closed}.

(i) $\Leftrightarrow$ (vi) $\Leftrightarrow$ (vii) $\Leftrightarrow$ (viii) $\Leftrightarrow$ (ii)  This is Lemma \ref{lem: Prufer local property}.
\end{myproof}

We now consider Conjecture \ref{conj: main}, which states that in order for $\Int_K(A)$ to be \Pruf{}, it is sufficient that $\Int_K(A)$ be integrally closed. The conjecture is true at least in the case given in Theorem \ref{thm: conjecture special cases}.

\begin{myproof}[of Theorem \ref{thm: conjecture special cases}]\mbox{}
Assume that $\Int_K(A)$ is integrally closed and that $\Int_K(A)_P = \Int_K(A_P)$ for all prime ideals $P$ of $D$. Since $\Int_K(A)$ is integrally closed, so is each localization $\Int_K(A)_P$. But, $\Int_K(A_P) = \Int_K(A)_P$ and $D_P$ is a DVR with finite residue rings, so by Theorem \ref{thm: Int_K(A,A')}, $\Int_K(A_P)$ is equal to its integral closure $\Int_K(A_P, A_P')$, which is \Pruf{} by Corollary \ref{cor: Lambda_n is Prufer}(2). By Lemma \ref{lem: Prufer local property}, $\Int_K(A)$ is \Pruf{}.
\end{myproof}

There exist examples \cite[Examples VI.4.15, VI.4.16]{CaCh} of \OurD{}-domains $D$ that are not Dedekind, and also examples \cite[Example VI.4.15]{CaCh} of domains $D$ such that $\Int(D)$ is \Pruf{}, but does not behave well under localization. Taking $A=D$ shows that neither $D$ being Dedekind nor $\Int_K(A)$ being well-behaved under localization is necessary for $\Int_K(A)$ to be \Pruf{}. Consequently, Conjecture \ref{conj: main} remains an open problem.

%%%%%%%%%%%%%%%%%%%%%
%%%%%%%%%%%%%%%%%%%%%
%%%%%%%%%%%%%%%%%%%%%
\section{A noncommutative example}\label{sec: quaternions}

Here, we give an example of a noncommutative $D$-algebra $A$ such that $\Int_K(A)$ is \Pruf{}. By Theorem \ref{thm: D semiprimitive}, $D$ cannot be semiprimitive. Our example uses quaternion algebras over the local domain $\Z_{(2)}$. 

\begin{Not} Let $D$ be an overring of $\Z$ in $\Q$. The ring of Hurwitz quaternions over $D$ is 
\begin{equation*}
D_{\HQ} = \{a_0+a_1\bfi+a_2\bfj+a_3\bfk \mid \text{either } a_i \in D \text{ for each $i$, or } a_i \in D+\tfrac{1}{2} \text{ for each $i$}\}.
\end{equation*}
Here, $\bfi$, $\bfj$, and $\bfk$ satisfy $\bfi^2=\bfj^2=-1$ and $\bfi \bfj = \bfk = -\bfj\bfi$. The ring $D_{\HQ}$ also contains the Hurwitz unit $\bfh = (1+\bfi+\bfj+\bfk)/2$, which is a root of $X^2-X+1$. The extended algebra $B = D_{\HQ} \otimes_D \Q$ of $D_{\HQ}$ is the division ring of rational quaternions
\begin{equation*}
B = \{a_0+a_1\bfi+a_2\bfj+a_3\bfk \mid a_i \in \Q \text{ for all $i$}\}.
\end{equation*}
It is well known that $\alpha = a_0+a_1\bfi+a_2\bfj+a_3\bfk \in B$ solves the polynomial $X^2-2a_0X + N(\alpha) \in \Q[X]$, where $N(\alpha) := a_0^2+a_1^2+a_2^2+a_3^2$ is the norm of $\alpha$. 
\end{Not}

\begin{Lem}\label{lem: 4^n}
Let $a, b, c, d, n \in \Z$ such that $n \geq 2$. Assume that $a^2+b^2+c^2+d^2 \equiv 0 \mod{4^n}$. Then, $a, b, c$, and $d$ are all even.
\end{Lem}
\begin{proof}
By assumption, $a^2+b^2+c^2+d^2 \equiv 0 \mod{4}$, which implies that $a, b, c$, and $d$ are either all odd or all even. If all four integers are odd, then $a^2+b^2+c^2+d^2 \equiv 4 \mod{8}$, which is a contradiction. So, $a, b, c$, and $d$ are all even.
\end{proof}

\begin{Lem}\label{lem: odd quaternions}
Let $D$ be an overring of $\Z$ in $\Q$.
\begin{enumerate}[(1)]
\item If $\alpha \in D_{\HQ}$, then $N(\alpha) \in D$. Thus, $\alpha$ is integral over $D$.
\item 
Let $D = \Z_{(2)}$ and let $a_0, a_1, a_2, a_3, e \in \Z$ be odd. Then, $(a_0+a_1\bfi+a_2\bfj+a_3\bfk)/(2e) \in D_{\HQ}$.
\end{enumerate}
\end{Lem}
\begin{proof}
(1) Let $\alpha \in D_{\HQ}$. If $\alpha = a_0+a_1\bfi+a_2\bfj+a_3\bfk$ with each $a_i \in D$, then the result is clear. If not, then each $a_i \in D+ \tfrac{1}{2}$, which means that $a_i = (1+2b_i)/2$ for some $b_i \in D$. Then, for each $i$, $(1+2b_i)^2 \in 1 + 4D$. So,
\begin{equation*}
(1+2b_0)^2 + (1+2b_1)^2 + (1+2b_2)^2 + (1+2b_3)^2 \in 4D,
\end{equation*}
and thus
\begin{equation*}
N(\alpha) = \tfrac{1}{4}\big((1+2b_0)^2 + (1+2b_1)^2 + (1+2b_2)^2 + (1+2b_3)^2\big) \in D,
\end{equation*}
as required.\\

(2) For each $0 \leq i \leq 3$, let $b_i \in \Z$ be such that $a_i = 2b_i + 1$. Then, $e$ is a unit of $\Z_{(2)}$, so \begin{equation*}
\dfrac{a_0 + a_1\bfi + a_2\bfj + a_3\bfk}{2e} = e^{-1}(\bfh + b_0 + b_1\bfi + b_2\bfj + b_3\bfk),
\end{equation*}
which is in $D_{\HQ}$.
\end{proof}

\begin{Thm}\label{thm: Z2 Hurwitz}
Let $D = \Z_{(2)}$ and $A = D_{\HQ}$. Then, $\IntQ(A)$ is a \Pruf{} domain.
\end{Thm}
\begin{proof}
We show that $A = A'$. Let $\alpha \in B$ be integral over $D$. We may write
\begin{equation*}
\alpha = \dfrac{a_0+a_1\bfi+a_2\bfj+a_3\bfk}{2^n e}, 
\end{equation*}
where $a_0, a_1, a_2, a_3, n, e \in \Z$, $n \geq 0$, and $e$ is odd. If $n = 0$, then $\alpha \in A$. So, assume that $n \geq 1$. Then, we may further assume that at least one $a_i$ is odd. Since $\alpha$ is integral over $D$, $N(\alpha) \in D$. Thus, there exist $c, d \in \Z$ with $d$ odd such that
\begin{equation*}%\label{eq: norm eq}
N(\alpha) = \dfrac{a_0^2+a_1^2+a_2^2+a_3^2}{2^{2n}e^2} = \dfrac{c}{d},
\end{equation*}
which is equivalent to
\begin{equation*}
(a_0^2+a_1^2+a_2^2+a_3^2)d = 4^n e^2 c.
\end{equation*}
Since $d$ is odd, this means that $a_0^2+a_1^2+a_2^2+a_3^2 \equiv 0 \mod{4^n}$. If $n \geq 2$, then each $a_i$ is even by Lemma \ref{lem: 4^n}, which is a contradiction. So, $n=1$. Since at least one $a_i$ is odd and $a_0^2+a_1^2+a_2^2+a_3^2 \equiv 0 \mod{4}$, each $a_i$ must be odd. Thus, $\alpha = (a_0+a_1\bfi+a_2\bfj+a_3\bfk)/(2e) \in A$, because $e$ and all the $a_i$ are odd. Therefore, $A=A'$, and $\IntQ(A)$ is \Pruf{} by Theorem \ref{thm: main}.

\end{proof}

\end{document}